\documentclass[12pt]{article}
\usepackage{amsmath,amsxtra,amssymb,latexsym,epsfig,amscd,amsthm,subfigure,fancybox,epsfig}
\usepackage[mathscr]{eucal}
\usepackage{graphicx}
\usepackage{epsfig} 
\usepackage{epstopdf}
\usepackage{cases}
\usepackage{subfig}
\newtheorem{thm}{Theorem}[section]

\newtheorem{lem}{Lemma}[section]
\newtheorem{prop}{Proposition}[section]
\theoremstyle{definition}

\theoremstyle{remark}

\newtheorem{example}{Example}[section]
\numberwithin{equation}{section}

\newcommand{\eps}{\varepsilon}

\newcommand{\F}{\mathcal{F}}

\newcommand{\E}{\mathbb{E}}

\newcommand{\N}{\mathbb{N}}
\newcommand{\PP}{\mathbb{P}}

\newcommand{\K}{\mathcal{K}}
\newcommand{\cB}{\mathcal{B}}

\newcommand{\R}{\mathbb{R}}

\newcommand{\cC}{\mathcal{C}}

\numberwithin{equation}{section}

\newcommand{\bed}{\begin{displaymath}}
\newcommand{\eed}{\end{displaymath}}
\newcommand{\bea}{\bed\begin{array}{rl}}
\newcommand{\eea}{\end{array}\eed}
\newcommand{\ad}{&\!\!\!\disp}
\newcommand{\aad}{&\disp}
\newcommand{\barray}{\begin{array}{ll}}
\newcommand{\earray}{\end{array}}

\def\disp{\displaystyle}

\newcommand{\1}{\boldsymbol{1}}

\def\bar{\overline}
\def\hat{\widehat}
\def\a.s{\text{\;a.s.\;}}

\begin{document}
\title{Classification of Asymptotic Behavior in A Stochastic SIR Model}
\author{N. T. Dieu,\thanks{Department of Mathematics, Vinh University,
182 Le Duan, Vinh, Nghe An, Vietnam, dieunguyen2008@gmail.com. The
research of this author was supported in part by the Foundation for Science and
Technology Development of Vietnam's Ministry of Education and
Training. No. B2015-27-15. The
author would also like
to thank Vietnam
Institute for Advance Study in Mathematics (VIASM) for supporting
and providing a fruitful research environment and hospitality.} \and
  D.H.  Nguyen,\thanks{Department of Mathematics, Wayne State University, Detroit, MI
48202, USA,
dangnh.maths@gmail.com.
The
research of this author was supported in part by the National Science Foundation under grant DMS-1207667. This research
was finished when the author was
visiting
the Institute for Advance Study in Mathematics (VIASM).
} \and
N.H. Du,\thanks{Department of Mathematics, Mechanics and
Informatics, Hanoi National University,
 334 Nguyen Trai, Thanh Xuan, Hanoi Vietnam, dunh@vnu.edu.vn. This research was
supported in part by
Vietnam National Foundation for Science and Technology Development   (NAFOSTED)
n$_0$  101.03-2014.58.}
 \and
G. Yin\thanks{Corresponding author: Department of Mathematics, Wayne State University, Detroit, MI
48202, USA, gyin@math.wayne.edu. The
research of this author was supported in part by the National Science Foundation 
under grant DMS-1207667. The author  thank VIASM for the support and hospitality during his visit.}
}
\maketitle

\begin{abstract}
This paper investigates
asymptotic behavior of a stochastic SIR epidemic model,
which is a system with degenerate diffusion.
It gives
sufficient conditions that are very close to the necessary conditions
for the permanence.
In addition,
 this paper develops ergodicity of the underlying system. It is
proved that  the transition probabilities  converge in
 total variation norm  to  the invariant measure.
Our result gives a precise characterization of the support of the invariant measure. 
Rates of convergence are also  ascertained.
It is shown that
the rate is not too far from
exponential in that the convergence speed is of the form of a polynomial of any degree.

\medskip



\medskip
\noindent{\bf Running Title.} Classification in A Stochastic SIR Model
\end{abstract}

\newpage
\setlength{\baselineskip}{0.28in}
\section{Introduction}\label{sec:int}
Since epidemic  models were  first introduced
  by  Kermack  and  McKendrick in \cite{KM,KM1},
 the study on
  mathematical models
has been flourished. Much attention has been devoted to
analyzing, predicting the spread, and designing controls of infectious diseases in host populations; see 
\cite{ABMSGS,BS, BC,  vC, GC, KRW, Kor, KM, KM1, SBKS, WZ} and the references therein.
One of classic epidemic models is
the SIR  (Susceptible-Infected-Removed) model
  that is suitable for modeling some diseases with permanent immunity such as rubella, whooping cough, measles, smallpox, etc.
  In the SIR model, a homogeneous host population is subdivided into three epidemiologically distinct types of individuals:
\begin{itemize}
\item (S): The susceptible class,
the class of those individuals who are capable of contracting the disease and becoming infective,
\item (I): the infective class,
the class of those individuals who are capable of transmitting the disease to others,
\item (R): the removed class,
the class of infected individuals who  are dead, or have recovered, and are permanently
immune, or are isolated.
\end{itemize}
If we denote by $S(t), I(t), R(t)$ the number of individuals at time $t$ in classes  (S), (I), and (R), respectively, the spread of infection can be formulated by the following deterministic system of differential equations:
\begin{equation}\label{e1.0}
\begin{cases}
dS(t)=\big(\alpha-\beta S(t)I(t)-\mu S(t))dt  \\
dI(t)=\big(\beta S(t)I(t)-(\mu+\rho+ \gamma) I(t))dt  \\
dR(t)=(\gamma I(t) - \mu R(t))dt,
\end{cases}
\end{equation}
where $\alpha$ is  the per capita birth rate of the population, $\mu$
is the per capita disease-free death rate and  $\rho$ is the excess per capita death rate of  infective class,  $\beta$ is the effective per capita contact rate, and $\gamma$ is per capita recovery rate of the infective individuals.
On the other hand, it is well recognized that the population is
always subject to random disturbances and
it is desirable to learn how randomness effects the models.
Thus, it is 
important to investigate stochastic epidemic models.
Jiang et al. \cite{JYS} investigated the asymptotic behavior
of global positive solution for the non-degenerate stochastic SIR model
\begin{equation}\label{e1.1a}
\begin{cases}
dS(t)=\big(\alpha-\beta S(t)I(t)-\mu S(t))dt+\sigma_1 S(t)dB_1(t)  \\
dI(t)=\big(\beta S(t)I(t)-(\mu +\rho+ \gamma) I(t))dt+\sigma_2 I(t)dB_2(t)  \\
dR(t)=(\gamma I(t) - \mu R(t))dt+\sigma_3 R(t)dB_3(t),
\end{cases}
\end{equation}
where $B_1(t)$, $B_2(t)$, and $B_3(t)$ are mutually independent
Brownian motions, $\sigma_1, \sigma_2, \sigma_3$ are the intensities
of the white noises.  However, in reality, the classes (S), (I), and (R)
are usually subject to the same random factors such as temperature, humidity, pollution and other extrinsic influences.
As a result, it is more plausible to assume that the random noise perturbing  the three classes is correlated.
If we assume that the Brownian motions $B_1(t)$, $B_2(t)$, and $B_3(t)$ are the same, we obtain the following model
\begin{equation}\label{e1.1bs}
\begin{cases}
dS(t)=\big(\alpha-\beta S(t)I(t)-\mu S(t))dt+\sigma_1 S(t)dB(t)  \\
dI(t)=\big(\beta S(t)I(t)-(\mu +\rho+ \gamma) I(t))dt+\sigma_2 I(t)dB(t)  \\
dR(t)=(\gamma I(t) - \mu R(t))dt+\sigma_3 R(t)dB(t),
\end{cases}
\end{equation}
which has been considered in \cite{LJX}.
Compared to \eqref{e1.1a}, \eqref{e1.1bs} is more difficult to deal with due to the degeneracy of the diffusion.
One of the important questions is concerned with
whether the
transition
to a disease free state or the disease state will
survive permanently.
For the deterministic model \eqref{e1.0}, the asymptotic behavior has been classified completely as follows.
If $\lambda_d=\dfrac{\beta\alpha}\mu-(\mu+\rho+\gamma)\leq0$, then the population tends to the disease-free equilibrium $(\dfrac{\alpha}\mu,0,0)$ while the population approaches an endemic equilibrium in case $\lambda_d>0$. In \cite{WZ}, similar results are given for a general epidemic model with reaction-diffusion in terms of basic reproduction numbers.
In \cite{LJX},  the authors attempted to answer the aforementioned question for
\eqref{e1.1bs} in case $\sigma_1>0$ and $\sigma_2>0$. By
using  Lyapunov-type functions, they provided some sufficient conditions for extinction or permanence as well as ergordicity for  the solution of system \eqref{e1.1bs}. Using the same methods, the extinction and permanence in some different stochastic SIR models have been studied in \cite{JJS,  YWS} etc. In practice,
 because of the randomness
 and the degeneracy of the diffusion, the model is much more difficult to deal with compared in contrast  to the deterministic counter part.
 Moreover,
 although one may assume the existence of appropriate Lyapunov function,
 it is fairly
  difficult to find an effective Lyapunov function in practice.
In other words, there has been no
decisive classification for stochastic SIR models that is similar to the deterministic case.

Our main goal in this paper is to
provide such a classification. We shall derive
a sufficient and almost necessary condition for permanence (as
well as ergodicity)  and extinction of the disease for the stochastic SIR model \eqref{e1.1bs} by using a value $\lambda$, which is similar to $\lambda_d$ in the deterministic model.
Note that such kind of results are obtained for a stochastic SIS model in \cite{GGHMP}.
However, the model studied there can be reduced to one-dimensional equation 
that is much easier to investigate.
The method used in \cite{GGHMP} cannot treat the stochastic SIR model \eqref{e1.1bs}. Estimation for the convergence rate is also not given in \cite{GGHMP}.  A more general method therefore need to be introduced.
The new method  can remove most assumptions in \cite{LJX}  as well as   can
 treat the case $\sigma_1>0, \sigma_2<0$,  which has not been taken into consideration in  \cite{LJX}.
 Note that the case $\sigma_1>0$ and $\sigma_2<0$ indicates the random factors have opposite effects to healthy individuals and infected ones.
 For instance,  patients with  tuberculosis or some other
 pulmonary disease do not endure well in cool  and humid weather while healthy people may
 be fine in such kind of weather.
In addition, individuals with a disease,
usually have weaker resistance to
some other kinds of disease.
Our new method is also suitable to  deal with other stochastic variants of \eqref{e1.0} such as models introduced in \cite{CKBW,JJS,YWS}, etc.

 The rest of the paper is arranged as follows. Section \ref{sec:thr} derives a threshold that is used
to classify the extinction and permanence of the disease. To establish the desired result, by considering
the dynamics on the boundary, we obtain a threshold $\lambda$ that enables us to determine the
asymptotic behavior of the solution. In particular, it is shown that if $\lambda < 0,$ the disease will decay in an exponential rate. In case $\lambda > 0$, the solution converges to a stationary distribution in total variation. It means that the disease is permanent. The rate of convergence is proved to be bounded above by any polynomial decay.
The ergodicity of the solution process is also proved.
Finally, Section \ref{sec:ex} is devoted to some discussion and comparison to
 existing results in the literature. Some numerical examples
 are  provided to illustrate our results.

\section{Threshold Between Extinction and Permanence}\label{sec:thr}
Let $(\Omega,\F,\{\F_t\}_{t\geq0},\PP)$ be a complete probability space with the filtration $\{\F_t\}_{t\geq 0}$ satisfying the usual condition,  i.e., it is increasing and right continuous while $\F_0$ contains all $\PP$-null sets. Let $B(t)$
be an $\F_t$-adapted,
Brownian motions. Because the dynamics of class of recover has no effect on the disease transmission dynamics,  we only consider the following system:
\begin{equation}\label{e1.1}\begin{cases}
dS(t)=[\alpha -\beta S(t)I(t)-\mu S(t)]dt+\sigma_1 S(t)dB(t),\\
dI(t)=[\beta S(t)I(t)-(\mu+\rho+\gamma)I(t)]dt+\sigma_2 I(t)dB(t).
\end{cases}
\end{equation}
Assume that $\sigma_1,\sigma_2\ne0$. By the symmetry of Brownian motions,  without loss of generality,  we suppose throughout this paper that $\sigma_1>0.$
Using standard arguments, it can be easily shown that for any positive initial value $(u,v)\in\R^{2,\circ}_+:=\{(u',v'): u'>0, v'>0\}$, there exists uniquely a global solution $(S_{u,v}(t), I_{u,v}(t)), t\geq0$ that remains in  $\R^{2,\circ}_+$ with probability 1 (see e.g.,  \cite{JYS}).
To obtain further properties of the solution,
we first consider  the equation on the boundary,
\begin{equation}\label{e2.1}
d\hat S(t)=(\alpha-\mu\hat S(t))dt+\sigma_1\hat S(t)dB(t).
\end{equation}
Let $\hat S_u(t)$ be the solution to \eqref{e2.1} with initial value $u$.
It follows from the comparison theorem \cite[Theorem 1.1, p.437]{IW} that $S_{u,v}(t)\leq \hat S_{u}(t)\; \forall t\geq0$ a.s. By solving the Fokker-Planck equation,  the process $\hat S_{u}(t)$ has
a unique stationary distribution with density
\begin{equation}\label{density}
f^*(x)=\dfrac{b^a}{\Gamma(a)} x^{-(a+1)}e^{\frac{-b}x}, x>0
\end{equation}
where $c_1=\mu+\frac{\sigma_1^2}{2}, a=\frac{2c_1}{\sigma_1^2}, b=\frac{2\alpha}{\sigma_1^2} $ and
 $\Gamma(\cdot)$ is the Gamma function.
By the strong law of large number we deduce that
\begin{equation}\label{e2.3'}
\lim\limits_{t\to\infty}\frac1t \int_0^t\hat S_{u}(s)ds=\int_{0}^\infty xf^*(x)dx:=\frac{\alpha}{\mu} \ \hbox{ a.s.}
\end{equation}
To proceed,
we
define the threshold as follows:
\begin{equation}
\label{lambda}
  \lambda:=\frac{\alpha\beta}{\mu}-\big(\mu+\rho+\gamma+\frac{\sigma^2_2}{2}\big).
\end{equation}

\subsection{Case 1: $\lambda<0$}
\begin{thm}\label{thm2.1}
If $\lambda<0$, then for any initial value $(S(0), I(0))=(u, v)\in\R_+^{2,\circ}$ we have
$\limsup\limits_{t\to\infty}\dfrac{\ln I_{u,v}(t)}t\leq\lambda$ a.s. and  the distribution of $S_{u,v}(t)$ converges weakly to
 the unique invariant probability measure $\mu^*$ with the density $f^*$.
 \end{thm}

\begin{proof}
Let $\hat I_{v}(t)$ be the solution to
\begin{equation*}
d \hat I(t)=\hat I(t)\big(-(\mu+\rho+\gamma)+\beta\hat S_{u}(t)\big)dt+\sigma_2 \hat I(t)dB(t), \;\;\hat I(0)=v.
\end{equation*}
where $\hat S(t)$ is the solution to \eqref{e2.1}.
By comparison theorem, $I_{u,v}(t)\leq\hat I_{v}(t)$ a.s. given that $\hat S(0)=S(0)=u, I(0)=\hat I(0)=v$.
In view of the It\^o formula and the ergodicity of $\hat S_{u}(t)$,
\begin{equation}\label{e2.6}
\begin{aligned}
\limsup\limits_{t\to\infty}\frac1t\ln\hat I_{v}(t)=&\limsup\limits_{t\to\infty}\bigg(\frac1t\int_0^t\Big(-(\mu+\rho+\gamma+\frac{\sigma^2_2}{2})+\beta\hat S_{u}(\tau)\Big)d\tau+\sigma_2\frac{B(t)}t\bigg)\\
=&\frac{\alpha\beta}{\mu}-\big(\mu+\rho+\gamma+\frac{\sigma^2_2}{2}\big)=\lambda<0\ \hbox{ a.s.}
\end{aligned}
\end{equation}
That is, $I_{u,v}(t)$ converges  almost surely to 0 at an exponential rate.

For any $\eps>0$, it follows from  \eqref{e2.6} that  there exists $t_0>0$ such that $\PP(\Omega_\eps)>1-\eps$ where
$$\Omega_\eps:=\Big\{ I_{u,v}(t)\leq \exp\Big\{\dfrac{\lambda t}2\Big\}\;\forall t\geq t_0\Big\}=\Big\{ \ln I_{u,v}(t)\leq \dfrac{\lambda t}2\;\forall t\geq t_0\Big\}.$$
Clearly, we can choose $t_0$ satisfying $-\dfrac{2\beta}{\lambda}\exp\Big\{\dfrac{\lambda t_0}2\Big\}<\eps$.
Let $\hat S_{u}(t), \,t\geq t_0$ be the solution to \eqref{e2.1} given that $\hat S(t_0)= S(t_0)$.
We have from the comparison theorem that $\PP\{S_{u,v}(t)\leq \hat S_{u}(t) \; \forall t\geq t_0\}=1$.
In view of the It\^o formula, for almost all $\omega\in\Omega_\eps$ we have
$$
\begin{aligned}
0\leq \ln \hat S_{u}(t)-\ln S_{u,v}(t)=&\alpha\int_{t_0}^t\Big[\dfrac1{\hat S_{u}(\tau)}-\dfrac1{ S_{u,v}(\tau)}\Big]d\tau+\beta\int_{t_0}^tI_{u,v}(\tau)d\tau.\\
\leq&\beta\int_{t_0}^t\exp\Big\{\dfrac{\lambda\tau}2\Big\}d\tau= -\dfrac{2\beta}{\lambda}\Big(\exp\Big\{\dfrac{\lambda t_0}2\Big\}-\exp\Big\{\dfrac{\lambda t}2\Big\}\Big)<\eps.
\end{aligned}
$$
As a result,
\begin{equation}\label{e2.8}
\PP\{|\ln S_{u,v}(t)-\ln \hat S_{u}(t))|>\eps\}\leq 1-\PP(\Omega_\eps)<\eps \;\forall t\geq t_0.
\end{equation}
Let $\nu^*$ be the distribution of a random variable $\ln X$ provided that $X$ admits $\mu^*$ as its distribution.
In lieu of proving that the distribution of $S(t)$ converges weakly to
 $\mu^*$, we
 claim an equivalent statement that
the distribution of $\ln S(t)$ converges weakly to
 $\nu^*$.
By the Portmanteau theorem (see \cite[Theorem 1]{BP}), it is sufficient to prove that for any
$g(\cdot):\R\mapsto\R$ satisfying $|g(x)-g(y)|\leq|x-y|$ and $|g(x)|<1\;\forall x,y\in\R$, we have
$$\E g(\ln  S_{u}(t))\to \bar g:=\int_{\R}g(x)\nu^*(dx)=\int_{0}^\infty g(\ln x)\mu^*(dx).$$

Since the diffusion given by \eqref{e2.1} is non-degenerate, it is well known that
the distribution of $\hat S_{u}(t)$ weakly converges to $\mu^*$ as $t\to\infty$ (see e.g.,  \cite{IK}).
Thus
\begin{equation}\label{e2.9}
\lim\limits_{t\to\infty}\E g(\ln \hat S_{u}(t))=\bar g.
\end{equation}
Note that
\begin{equation}
\label{e2.10} \barray
\disp |\E g(\ln S_{u,v}(t))-\bar g|\ad \leq |\E g(\ln S_{u,v}(t))-\E g(\ln \hat S_{u}(t))|+|\E g(\ln \hat S_{u}(t))-\bar g|\\
\ad \leq \eps\PP\{|\ln S_{u,v}(t)-\ln \hat S_{u}(t)|\leq\eps\}+2\PP\{|\ln S_{u,v}(t)-\ln \hat S_{u}(t)|>\eps\}\\
\aad \qquad +|\E g(\ln \hat S_{u}(t))-\bar g|.
\earray\end{equation}
Applying \eqref{e2.8} and \eqref{e2.9} to \eqref{e2.10} yields
$$\limsup\limits_{t\to\infty}|\E g(\ln S_{u,v}(t))-\bar g|\leq3\eps.$$
Since $\eps$ is taken arbitrarily, we obtain the desired conclusion.
The proof is complete.
\end{proof}

\subsection{Case 2: $\lambda>0$}
We now focus on the case $\lambda>0$.
Let $P(t, (u,v),\cdot)$ be the transition probability of $(S_{u,v}(t), I_{u,v}(t))$.
To obtain properties of $P(t, (u,v),\cdot)$, we first rewrite equation \eqref{e1.1} in Stratonovich's form
\begin{equation}\label{e2.11ss}
\begin{cases}
dS(t)=[\alpha-c_1S(t) -\beta S(t)I(t)]dt+\sigma_1S(t)\circ dB(t),\\
dI(t)=[-c_2I(t)+\beta S(t)I(t)]dt+\sigma_2 I(t)\circ dB(t).
\end{cases}
\end{equation}
where $c_1=\mu+\frac{\sigma_1^2}{2}, c_2=\mu+\rho+\gamma+\frac{\sigma_2^2}2.$
Put
$$A(x, y)=\left(\begin{array}{l}\alpha-c_1x -\beta xy\\
\hspace{.1cm}-c_2y+\beta xy
\end{array}\right)\,\mbox{ and }\,
B(x, y)=\left(\begin{array}{l}\sigma_1x\\
\sigma_2y
\end{array}\right).$$
To proceed, we first
 recall the notion of Lie bracket.
If $\Phi(x,y)=(\Phi_1, \Phi_2)^\top$ and $\Psi(x,y)=(\Psi_1, \Psi_2)^\top$ are vector fields on $\R^2$ then the Lie bracket $[\Phi,\Psi]$ is a vector field given by
$$[\Phi,\Psi]_j(x,y)=\Big(\Phi_1 \frac{\partial \Psi_j}{\partial x}(x,y)-\Psi_1 \frac{\partial \Phi_j}{\partial x}(x,y)\Big)+\Big(\Phi_2 \dfrac{\partial \Psi_j}{\partial y}(x,y)-\Psi_2 \frac{\partial \Phi_j}{\partial y}(x,y)\Big), \; j=1,2.$$
Denote by $\mathcal L(x, y)$ the Lie algebra generated by $A(x, y), B(x, y)$ and $\mathcal L_0(x, y)$ the ideal in $\mathcal L(x, y)$ generated by $B$. We have the following lemma.

\begin{lem}\label{asp3.1} For $\sigma_1>0, \sigma_2\ne 0$, the H\"{o}rmader condition holds for the diffusion \eqref{e2.11ss}.
To be more precise, we have $\dim\mathcal L_0(x, y)=2$ at every $(x, y)\in\R^{2,\circ}_+$ or equivalently, the set of vectors $ B, [A, B], [A, [A, B]], [B, [A, B]],\dots$ spans $\R^2$ at every $(x, y)\in\R_+^{2,\circ}$. As a result, the transition probability $P(t, (u,v),\cdot)$ has smooth density
$p(t, u,v, u', v')$.
\end{lem}

\begin{proof}
This lemma has been proved in \cite{LJX} for the case $\sigma_2>0$.
Assume that $r=-\dfrac{\sigma_2}{\sigma_1}>0$.
It is easy to obtain
\begin{align*}
  C:=&\dfrac1{\sigma_1}B(x,y)=\begin{pmatrix}
    x\\ - ry
  \end{pmatrix},
\\
  D:=&[A, C](x,y)=\begin{pmatrix}
    \alpha -r\beta xy\\ -\beta xy
  \end{pmatrix},
\\
E:=&[C,D](x,y)=\begin{pmatrix}
   -\alpha +r^2\beta xy\\ -\beta xy
  \end{pmatrix},
\\
  F:=&[C, E](x,y)=\begin{pmatrix}
  \alpha-r^3\beta xy\\ -\beta xy
  \end{pmatrix}.
\end{align*}
Since $\det(D,F)=0$ only if $r^2=1$ or $r=1$ (since $r>0$).
When $r=1$, solving $\det(D, E)=0$ obtains
$\beta xy=\alpha$
which implies
$$\det(C,D)=\left|\begin{array}{cc}x&0\\-y &-\alpha\end{array}\right|\ne0.$$
As a result, $B, D, E, F$  span $\R^2$ for all $(x, y)\in\R_+^{2,\circ}$. The lemma is proved.
\end{proof}
In order to describe the support of the invariant measure $\pi^*$ (if it exists) and to prove the ergodicity of \eqref{e1.1}, we need to investigate the following control system on $\R^{2,\circ}$
\begin{equation}\label{e3.2}
\begin{cases}
   \dot u_{\phi}(t)=\sigma_1 u_{\phi}(t)\phi(t)+\alpha -\beta u_{\phi}(t)v_{\phi}(t)-c_1 u_{\phi}(t),\\
\dot v_{\phi}(t)=\sigma_2 v_{\phi}(t)\phi(t)+\beta u_{\phi}(t)v_{\phi}(t)-c_2v_\phi(t),
\end{cases}
\end{equation}
where  $\phi$ is taken from the set of piecewise continuous real-valued functions defined on $\R_+$.
Let $(u_\phi(t, u,v),$ $ v_\phi(t, u, v))$ be the solution to equation \eqref{e3.2} with control $\phi$ and initial value $(u,v)$.
Denote by ${\cal O}^+(u, v)$ the reachable set from $(u, v)\in\R^{2,\circ}_+$, that is the set of $(u', v')\in\R^{2,\circ}_+$ such that there exists a $t\geq0$ and a control $\phi(\cdot)$ satisfying
$u_\phi(t, u, v)=u', v_\phi(t, u, v)=v'$.
We now recall some concepts introduced in  \cite{WK}.
Let $X$ be a subset of $\R^{2,\circ}_+$ satisfying the property that for any $w_1, w_2\in X$,
$w_2\in \bar{{\cal O}^+(w_1)}$.
Then there is a unique maximal set $Y\supset X$ such that this property still holds for $Y$. Such a $Y$ is called a control set.
A control set $W$ is said to be invariant if $\bar{{\cal O}^+(w)}\subset\bar W$ for all $w\in W$.

Putting $r:=\frac{-\sigma_2}{\sigma_1}$ and $z_\phi(t)=u^r_\phi(t)v_\phi(t)$, we have an equivalent system
\begin{equation}\label{e3.3}
\left\{\begin{array}{l}\dot u_\phi(t)=\sigma_1\phi(t)u_\phi(t)+g(u_\phi(t),z_\phi(t)),\\
\dot z_\phi(t)=h(u_\phi(t), z_\phi(t)),
\end{array}\right.
\end{equation}
where
$$g(u,z )=-c_1e^u+\alpha -\beta zu^{1-r},$$
and
\begin{align*}
h(u, z)=u^{-r}z\Big[-(c_1r+c_2)u^r+\beta u^{1+r}+\alpha ru^{r-1}-\beta r z\Big].
\end{align*}

\begin{lem} \label{lem2.3}
For the control system \eqref{e3.3}, the following claims hold
  \begin{enumerate}
    \item For any $u_0, u_1, z_0>0$ and $\eps>0$, there exists a control $\phi$ and $T>0$ such that
$u_\phi(T, u_0, z_0)=u_1$, $|z_\phi(T, u_0, z_0)-z_0|<\eps$.
    \item For any $0<z_0<z_1$, there is a $u_0>0$, a control $\phi$, and $T>0$ such that $z_\phi(T, u_0, z_0)=z_1$ and that $u_\phi(t, u_0, z_0)=u_0 \;\forall\,0\leq t\leq T$.
    \item Let $d^*=\inf\limits_{u>0}\{-(c_1r+c_2)u^{r}+\beta u^{1+r}+\alpha ru^{r-1}\}.$
        \begin{enumerate}
          \item If $d^*\leq 0$ then for any $z_0>z_1$, there is  $u_0>0$, a control $\phi$, and $T>0$ such that $z_\phi(T, u_0, z_0)=z_1$ and that $u_\phi(t, u_0, z_0)=u_0 \;\forall\,0\leq t\leq T$.
          \item Suppose that $d^*>0$ and $z_0>c^*:=\dfrac{d^*}{\beta r}$. If $c^*<z_1<z_0$, there is $u_0>0$ and  a control $\phi$ and $T>0$ such that $z_\phi(T, u_0, z_0)=z_1$ and that $u_\phi(t, u_0, z_0)=u_0 \;\forall\, 0\leq t\leq T$. However,  there is no control $\phi$ and $T>0$ such that
$z_\phi(T, u_0, z_0)<c^*$.
        \end{enumerate}
        \end{enumerate}
\end{lem}

\begin{proof}
Suppose that $u_0<u_1$ and let $\rho_1=\sup\{|g(u, z)|, |h(u, z)|: u_0\leq u\leq u_1, |z-z_0|\leq\eps\}.$ We choose  $\phi(t)\equiv\rho_2$ with $\Big(\dfrac{\sigma_1\rho_2u_0}{\rho_1}-1\Big)\eps\geq u_1-u_0$. It is easy to check that with this control, there is $0\leq T\leq\frac{\eps}{\rho_1}$ such that $u_\phi(T, u_0, z_0)=u_1$, $|z_\phi(T, u_0, z_0)-z_0|<\eps$.
If $u_0>u_1$, we can construct $\phi(t)$ similarly. Then the
claim 1 is proved.

By choosing $u_0$
to be sufficiently
large, there is a $\rho_3>0$ such that
$h(u_0, z)>\rho_3\;\forall z_0\leq z\leq z_1$. This property, combining with \eqref{e3.3}, implies the existence of a feedback control $\phi$ and $T>0$ satisfying that $z_\phi(T, u_0, z_0)=z_1$ and that $u_\phi(t, u_0, z_0)=u_0,\;\forall\, 0\leq t\leq T$.

We now prove
 claim 3. If $r<0$ then $$\lim\limits_{u\to0}\big[-(c_1r+c_2)u^{r}+\beta u^{1+r}+\alpha ru^{r-1}]=-\infty$$ and
 $$\lim\limits_{u\to0}\big[-(c_1r+c_2)u^{r}+\beta u^{1+r}+\alpha ru^{r-1}]=0 \ \hbox{ if } \ r>1.$$ As a result, $d^*\leq0$ if $r\notin(0,1]$ which implies that for any $z_0>z_1$, we choose $u_0$ such that  $\sup_{z\in[z_1,z_0]}h(u_0, z)<0$, which implies that there is a feedback control $\phi$ and $T>0$ satisfying $z_\phi(T, u_0, z_0)=z_1$ and  $u_\phi(t, u_0, z_0)=u_0 \;\forall 0\leq t\leq T$.

 If $r\in(0,1]$  there exists $u_0$ such that $-(c_1r+c_2)u_0^{r}+\beta u_0^{1+r}+\alpha ru_0^{r-1}=d^*$. If $d^*\leq 0$, then for any $ z_0>z_1>0$ we have $\sup_{z\in[z_1,z_0]}h(u_0, z)\leq u_0^{-r}\sup_{z\in[z_1,z_0]}\{-\beta r z^2\}<0$ which implies the desired claim.
\par Consider the remaining case when $r\in(0,1]$ and $d^*>0$.  First, assume $c^*<z_1<z_0$. Let $u_0$ satisfy
$-(c_1r+c_2)u_0^{r}+\beta u_0^{1+r}+\alpha r{u_0}e^{r-1}=d^*=\beta rc^*$. Hence
$$
\begin{aligned}
\sup\limits_{z\in[z_1,z_0]}\{h(u_0, z)\}=&u_0^{-r}\sup\limits_{z\in[z_1,z_0]}\Big\{z\Big(-(c_1r+c_2)u_0^{r}+\beta u_0^{1+r}+\alpha ru_0^{r-1}-\beta r {z}\Big)\Big\}\\
=&-\beta ru_0^{-r}z_1(c^*-{z_1})<0.
\end{aligned}
$$
Thus, there is a feedback control $\phi$ and $T>0$ satisfying $z_\phi(T, u_0, z_0)=z_1$ and  $u_\phi(t, u_0, z_0)=u_0 \;\forall 0\leq t\leq T$.
The final assertion follows from the fact that $h(u,c^*)\geq0$ for all $u\in\R$.
\end{proof}

To obtain the convergence in total variation norm and to estimate the convergence rate,
we aim to  apply \cite[Theorem 3.6, p. 235]{JR}. In order to do that, we construct a  function $V:\R^{2,\circ}_+\to[1,\infty)$ satisfying that
$$\E V\big(S_{u,v}(t^*), I_{u,v}(t^*)\big)\leq V(u,v)-\kappa_1 V^\gamma(u,v)+\kappa_2\1_{\{(u,v)\in K\}}$$
for some petite set $K$ and some $\gamma\in(0,1), \kappa_1,\kappa_2>0,, t^*>1$.
Recall that a set $K$ is said to be petite with respect to the Markov chain $S_{u,v}(nt^*), I_{u,v}(nt^*),n\in\N$ if
 there exists a measure $\psi$ with $\psi(\R^{2,\circ}_+)>0$ and a probability distribution $\nu(\cdot)$  concentrated on $\N$ such that   $$
    \K(u,v,Q):=\sum_{n=1}^\infty P(nt^*,u,v,Q)\nu(n)\geq \psi(Q)\;  \forall (u,v)\in K,\; Q\in\cB(\R_+^{2,\circ}).
  $$
We also have to prove that the skeleton Markov chain
$\big(S_{u,v}(nt^*), I_{u,v}(nt^*)\big), n\in\N$ is irreducible and aperiodic.
We refer to \cite{EN} or \cite{MT} for the definitions and properties of irreducibility, aperiodicity, as well as petite sets.
The estimation of the convergence rate is divided into some lemmas and propositions.

\begin{lem}\label{lem2.1} For any  $0<p^*<\min\{\frac{2\mu}{\sigma_1^2},\frac{2(\mu+\rho+\gamma)}{\sigma_2^2}\}$. Let $U(u,v)=(u+v)^{1+p^*}+u^{-\frac{p^*}2}.$ There exists positive constants $K_1, K_2$ such that
\begin{equation}
   e^{K_1t}\E(U(S_{u,v}(t), I_{u,v}(t))) \leq U(u,v)
  +\frac{K_2(e^{K_1t}-1)}{K_1}.
\end{equation}

\end{lem}

\begin{proof}
Consider the Lyapunov function
$U(u,v)=(u+v)^{1+p^*}+u^{-\frac{p^*}2}.$ By directly calculating the differential operator $LU(u,v)$ associated with equation \eqref{e1.1}, we have
\begin{align}\label{e2.16}
  LU(u, v)=&
  (1+p)(u+v)^{p^*}(\alpha-\mu u-(\mu+\rho+\gamma)v)
+\frac{(1+p^*)p^*}{2}(u+v)^{p-1}(\sigma_1 u+\sigma_2v)^2
\notag\\&-\frac{p^*}2u^{-\frac{p^*}2-1}(\alpha-\beta uv-\mu u)+\frac{p^*(2+p^*)}8\sigma_1^2 u^{-\frac{p^*}2}
\notag\\
=&(1+p^*)\alpha(u+v)^{p^*}
-(1+p^*)(u+v)^{p^*-1}\Big[(\mu-\frac{p^*}2\sigma_1^2) u^2+(\mu+\rho+\gamma-\frac{p^*}2\sigma_2^2)v^2)
\notag\\&
+(2\mu+\rho+\gamma-p^*\sigma_1\sigma_2)uv\Big]-p^*\alpha u^{-\frac{2+p^*}2}+\frac{\beta p^*}{2}  u^{-\frac{p^*}2}v+\frac{p^*}2\Big[\frac{(2+p^*)\sigma_1^2 }4+\mu \Big] u^{-\frac {p^*}2}.
\end{align}
By Young's inequality, we have
\begin{equation}\label{e2.17}
  u^{-\frac {p^*}2}v\leq \frac{3p^*}{4+3p^*}u^{-\frac{4+3p^*}6}+\frac{4}{4+3p^*} v^{\frac{4+3p^*}4}.
\end{equation}
Choose a number $K_1$ satisfying
 $0<K_1<\min\{\mu-\frac{p^*}2\sigma_1^2, \mu+\rho+\gamma-\frac{p^*}2\sigma_2^2\}.$
From \eqref{e2.17}  \eqref{e2.16}, we obtain
$
K_2=\sup_{u, v\in\R_+}\{LU(u,v)+K_1U(u,v)\}<\infty.
$
As a result,
\begin{equation}\label{eU}
LU(u,v)\leq K_2-K_1U(u,v)\;\forall (u,v)\in\R^2_+.
\end{equation}
 For $n\in\N$, define the stopping time
\begin{equation*}
  \eta_n=\inf\{t\geq 0: U(S_{u,v}(t), I_{u,v}(t))\geq n\}.
\end{equation*}
Then It\^{o}'s formula and \eqref{eU} yield that
\begin{equation*}
  \begin{aligned}
  \E( e^{K_1(t\wedge\eta_n)}&U(S_{u,v}(t\wedge\eta_n), I_{u,v}(t\wedge\eta_n)))\\
  &\leq U(u,v)+\E\int_0^{t\wedge\eta_n}e^{K_1\tau}\Big(LU(S_{u,v}(\tau), I_{u,v}(\tau))+K_1 U(S_{u,v}(\tau), I_{u,v}(\tau))\Big)d\tau
  \\
  &\leq U(u,v)+\frac{K_2(e^{K_1t}-1)}{K_1}.
  \end{aligned}
\end{equation*}
  By letting $n\to \infty$ we obtain from Fatou's lemma that
\begin{equation}
  \E e^{K_1t}(U(S_{u,v}(t), I_{u,v}(t))) \leq U(u,v)
  +\frac{K_2(e^{K_1t}-1)}{K_1}.
\end{equation}
The lemma is proved. \end{proof}

\begin{lem}\label{lem2.4}
  There are positive constants $K_3, K_4$ such that, for any $t\geq 1$ and $A\in\F$
  \begin{equation}
    \E \left([\ln I_{u,v}(t)]_-^2\1_{A}\right)\leq [\ln v]_-^2\PP(A)+K_3\sqrt{\PP(A)}t[\ln v]_-+K_4t^2\sqrt{\PP(A)},
  \end{equation}
  where
  $[\ln x]_-=\max\{0, -\ln x\}.$
\end{lem}

\begin{proof}
  We have
  \begin{align*}
  -\ln I_{u,v}(t)&=-\ln I_{u,v}(0)-\beta\int_0^tS_{u,v}(\tau)d\tau+c_2t+\sigma_2 B(t)
  \\& \leq-\ln v+c_2t+\sigma_2 |B(t)|,
  \end{align*}
  where $c_2=\mu+\rho+\gamma+\frac{\sigma_2^2}2.$ Therefore,
  $$[\ln I_{u,v}(t)]_-\leq [\ln v]_-+c_2t+\sigma_2 |B(t)|.$$
  This implies that
  \begin{align*}
    [\ln I_{u,v}(t)]^2_-\1_A\leq& [\ln v]^2_-\1_A+\big(c^2_2t^2+\sigma^2_2 B^2(t)\big)\1_A+2c_2t[\ln v]_-\1_A
    \\&+2\sigma_2|B(t)|\1_A[\ln v]_-+2c_1t\sigma_2|B(t)|\1_A.
  \end{align*}
By using H\"{o}lder inequality, we obtain
$$\E|B(t)|\1_A\leq\sqrt{\E B^2(t)\PP(A)}\leq \sqrt{t\PP(A)}\leq t\sqrt{\PP(A)}.$$
Taking expectation both sides and using the estimate above, we have
  \begin{equation*}
   \E [\ln I_{u,v}(t)]^2_-\1_A\leq[\ln v]_-^2\PP(A)+K_3t\sqrt{\PP(A)}[\ln v]_-+K_4t^2\sqrt{\PP(A)},
     \end{equation*}
for some positive constants $K_3, K_4.$
\end{proof}

We now, choose $\eps\in(0,1)$ satisfying
\begin{equation}\label{e2.35}
  -\frac{3\lambda}2(1-\eps)+K_3\sqrt{\eps}<-\lambda
\text{ and }
-\frac{3\lambda}4(1-\eps)+2K_3\sqrt{\eps}<-\frac{\lambda}2.
\end{equation}
Choose $H$ so large that
\begin{equation}\label{e2.36}
\beta H-2c_2\geq 2+\lambda; \;\;\exp\left\{-\frac{\beta H-2c_2}{2\sigma_2^2}\right\}<\frac{\eps}2\;\;\text{and}\;\;\exp\left\{-\frac{\lambda(\beta H-c_2)}{4\sigma_2^2}\right\}<\frac{\eps}2.
\end{equation}

\begin{lem}\label{lem2.5}
  For $\eps$ and $H$ chosen as above, there is $\delta\in(0,1)$ and $T^*>1$ such that
  \begin{equation}
    \PP\{\ln v+\frac{3\lambda t}{4}\leq \ln I_{u,v}(t)<0\;\;\forall\, t\in [T^*, 2T^*]\}\geq 1-\eps
  \end{equation}
  for all $u\in[0, H], v\in(0,\delta].$
\end{lem}

\begin{proof}
\par Let $\theta\in(0, 1)$ such that
\begin{equation}\label{e2.30}
\dfrac{\beta\alpha}{\mu+\beta\theta}-c_2\geq \dfrac{11\lambda}{12}.
\end{equation}
Let $\tilde{S}_{u}(t)$ be the solution with initial value $u$ to
\begin{equation}\label{e2.31}
  d\tilde{S}(t)=[\alpha-(\beta\theta+\mu)\tilde{S}(t)]dt+\sigma_1\tilde{S}(t)dB(t)
\end{equation}
  Similar to \eqref{e2.3'},
  $$\PP\left\{\lim_{t\to\infty}\frac1t\int_0^t\tilde{S}_u(\tau)d\tau=\frac{\alpha}{\mu+\beta\theta}\right\}=1\;\;\forall\, u\in [0.\infty).$$
In view of the strong law of large numbers for martingales, $\PP\{\lim_{t\to\infty}\dfrac{B(t)}t=0\}=1$.
Hence, there exists $T^*>1$ such that
\begin{equation}\label{e2.37bs}
  \PP\Big\{\frac{\sigma_1 B(t)}t\geq\frac{-\lambda}{12}\;\;\forall \, t\geq T^*\Big\}\geq 1-\frac{\eps}3
\end{equation}
and
$$\PP\left\{\frac1t\int_0^t\tilde{S}_0(\tau)d\tau\geq\frac{\alpha}{\mu+\beta\theta}-\frac{\lambda}{12\beta}\;\;\forall\, t\geq T^*\right\}\geq1-\frac{\eps}3.$$
By the uniqueness of solutions to \eqref{e2.31},
$$\PP\left\{\tilde{S}_0(t)\leq\tilde{S}_u(t)\;\;\forall t\geq 0\right\}=1\;\;\forall\, u\geq 0.$$
Hence,
\begin{equation}\label{e2.38bs}
\PP\left\{\frac1t\int_0^t\tilde{S}_u(\tau)d\tau\geq\frac{\alpha}{\mu+\beta\theta}-\frac{\lambda}{12\beta}\;\;\forall\, t\geq T^*\right\}\geq1-\frac{\eps}3.
\end{equation}
Similar to \cite[Lemmas 3.1, 3.2]{DY}, it can be shown that there exists $\delta\in (0,\theta),$
\begin{equation}\label{e2.39bs}
  \PP\left\{\zeta_{u,v}\leq2T^*\right\}\leq\frac{\eps}3,\;\;\forall v\leq\delta, \;\;u\in [0, H] \text{ where } \zeta_{u,v}=\inf\{t\geq 0: I_{u,v}(t)\geq\theta\}.
\end{equation}
Observe also that
\begin{equation}\label{e2.40}
  \PP\left\{S_{u,v}(t)\geq \tilde{S}_u(t)\;\;\forall\, t\leq\zeta_{u,v}\right\}=1
\end{equation}
which we have from the comparison theorem.  From \eqref{e2.30}, \eqref{e2.37bs},\eqref{e2.38bs}, \eqref{e2.39bs} and \eqref{e2.40} we can be show  that with probability greater than $1-\eps,$ for all $t\in[T^*, 2T^*],$
\begin{align*}
  \ln\theta&\geq\ln(I_{u,v}(t))=\ln v+\beta\int_0^tS_{u,v}(\tau)d\tau-c_2t+\sigma_1 B(t)
  \\&
  \geq \ln v+\frac{\beta\alpha t}{(\mu+\beta\theta)}-\frac{\lambda t}{12}-c_2t-\frac{\lambda t}{12}
  \geq \ln v+\frac{3\lambda}{4}t.
\end{align*}
The proof is complete.
\end{proof}

\begin{prop}\label{prop2.3} Assume $\lambda>0$.
  Let $\delta, H$ and $T^*$ be  as in Lemma \ref{lem2.5}. There exists $K_5,$ independent of $T^*$ such that
  \begin{equation}
    \E[\ln I_{u,v}(t)]_-^2\leq [\ln v]_-^2-\lambda t[\ln v]_-+K_5 t^2
  \end{equation}
  for any $v\in (0,\infty),\;\; 0\leq u\leq H,\; t\in[T^*, 2T^*].$
\end{prop}

\begin{proof}
  First, consider $v\in (0,\delta],\;\; 0\leq u\leq H.$ We have $\PP(\Omega_{u,v})\geq1-\eps$ where
  $$\Omega_{u,v}=\big\{\ln v+\dfrac{3\lambda t}4\leq \ln I_{u,v}(t)<0\;\;\;\forall\, t\in[T^*, 2T^*]\big\}.$$
  In $\Omega_{u,v}$ we have
  $$-\ln v-\frac{3\lambda t}4\geq-\ln I_{u,v}(t)> 0.$$
  Hence,
  $$0\leq [\ln I_{u,v}(t)]_-\leq [\ln v]_--\frac{3\lambda t}4\;\;\forall\, t\in[T^*, 2T^*].$$
  As a result
  $$[\ln I_{u,v}(t)]_-^2\leq [\ln v]_-^2-\frac{3\lambda t}2[\ln v]_-+\frac{9\lambda^2 t^2}{16}\;\;\forall\, t\in[T^*, 2T^*].$$
  Which implies that
  \begin{equation}\label{e2.37}
    \E\left[\1_{\Omega_{u,v}}[\ln I_{u,v}(t)]^2_-\right]\leq\PP(\Omega_{u,v})[\ln v]_-^2-\frac{3\lambda t}{2}\PP(\Omega_{u,v})[\ln v]_-+\frac{9\lambda^2 t^2}{16}\PP(\Omega_{u,v}).
  \end{equation}
 In $\Omega_{u,v}^c=\Omega\backslash \Omega_{u,v},$ we have from Lemma \ref{lem2.4} that
 \begin{equation}\label{e2.38}
    \E\left[\1_{\Omega^c_{u,v}}[\ln I_{u,v}(t)]^2_-\right]\leq\PP(\Omega^c_{u,v})[\ln v]_-^2+K_3t \sqrt{\PP(\Omega^c_{u,v})}[\ln v]_-+K_4t^2 \sqrt{\PP(\Omega^c_{u,v})}.
 \end{equation}
 Adding \eqref{e2.37} and \eqref{e2.38} side by side, we obtain
  \begin{equation}\label{e2.39}
    \E[\ln I_{u,v}(t)]^2_-\leq[\ln v]_-^2+\Big(-\frac{3\lambda}2(1-\eps)+K_3\sqrt{\eps}\Big)t[\ln v]_-+\Big(\frac{9\lambda^2}{16}+K_4\Big)t^2 .
 \end{equation}
 In view of \eqref{e2.35} we deduce
  \begin{equation*}
    \E[\ln I_{u,v}(t)]^2_-\leq[\ln v]_-^2-\lambda t[\ln v]_-+\Big(\frac{9\lambda^2}{16}+K_4\Big)t^2 .
 \end{equation*}
 Now, for $v\in [\delta, \infty)$ and $0\leq u\leq H,$ we have from Lemma \ref{lem2.4} that
 \begin{align*}
    \E[\ln I_{u,v}(t)]^2_-&\leq[\ln v]_-^2+K_3 t[\ln v]_-+K_4t^2
   \\&\leq |\ln\delta|^2+K_3 t|\ln\delta|+K_4t^2.
 \end{align*}
 Letting $K_5$ sufficiently large such that $K_5>\frac{9\lambda^2}{16}+K_4$ and
 $|\ln\delta|^2+K_3 t|\ln\delta|+K_4t^2\leq K_5 t^2\;\;\forall t\in [T^*, 2T^*]$,
 we obtain the desired result.
\end{proof}
\begin{prop}\label{prop2.4} Assume $\lambda>0$.
  There exists $K_6>0$  such that
  $$\E[\ln I_{u,v}(2T^*)]_-^2\leq [\ln v]_-^2-\frac{\lambda}2[\ln v]_-+K_6 {T^*}^2$$
  for $v\in(0,\infty), u>H.$
\end{prop}
\begin{proof}
  First, consider $v\leq\exp\{-\frac{\lambda T^*}2\}.$ Defined the stopping time
  $$\xi_{u,v}=T^*\wedge\inf\{t>0: S_{u,v}(t)\leq H\}.$$
Let $$\Omega_1=\big\{\sigma_2 B(2T^*)-\frac{(\beta H-2c_2)T^*}{2}\leq 1\big\}$$
and $$\Omega_2=\Big\{\sigma_2 B(t)-(\beta H-c_2)t\leq\frac{\lambda}8\,\forall\,t\in[0,2T^*]\Big\}.$$
By the exponential martingale inequality \cite[Theorem 7.4, p. 44]{MAO},
\begin{equation}\label{e2.41}
  \PP(\Omega_1)\geq 1-\exp\Big\{-\frac{\beta H-2c_2}{2\sigma_2^2}\Big\}\geq 1-\frac{\eps}{2}.
\end{equation}
and
$$\PP(\Omega_2)\geq 1-\exp\Big\{-\frac{\lambda(\beta H-c_2)}{4\sigma_2^2}\Big\}\geq 1-\frac{\eps}2.$$
Let
$$\Omega_3=\Omega_1\cap \{\xi_{u,v}= T^*\};\;\;\;\Omega_4=\{-\ln I_{u,v}(\xi_{u,v})\leq -\ln v+\frac{\lambda}8\}\cap\{\xi_{u,v}<T^*\}; \;\;\;\Omega_5=\Omega\backslash(\Omega_3\cup\Omega_4).$$
 If $\omega\in \Omega_3,$ we have
\begin{align*}
  -\ln I_{u,v}(2T^*)&=-\ln v-\int_0^{2T^*}(\beta S_{u,v}(\tau)-c_2)d\tau+\sigma_2 B(2T^*)
  \\&\leq-\ln v-\int_0^{T^*}(\beta S_{u,v}(\tau)-c_2)d\tau+ \int_{T^*}^{2T^*}c_2d\tau+\sigma_2B(2T^*)
  \\&\leq-\ln v-T^*(\beta H-c_2)+ T^*c_2+\sigma_2B(2T^*)
\\&\leq-\ln v-T^*(\beta H-2c_2)+\sigma_2B(2T^*)
\\& \leq-\ln v-\frac{T^*(\beta H-2c_2)}{2}+ 1\;\;\;(\text{in view of \eqref{e2.41}})
\\&\leq -\ln v-\frac{\lambda T^*}2 \text{ (by } \eqref{e2.36}).
\end{align*}
If $v\leq \exp\{-\frac{\lambda T^*}{2}\}$ then $-\ln v-\frac{\lambda T^*}{2}>0.$ Therefore
$$[\ln I_{u,v}(2T^*)]_-\leq -\frac{\lambda T^*}{2}+[\ln v]_-.$$
Squaring and then multiplying by $\1_{\Omega_3}$ and then taking expectation both sides, we yield
\begin{equation}\label{e2.42}
  \E\left( [\ln I_{u,v}(2T^*)]^2_-\1_{\Omega_3}\right)\leq [\ln v]_-^2\PP(\Omega_3)-\lambda T^*[\ln v]_-\PP(\Omega_3)+\frac{\lambda^2 {T^*}^2}{4}.
\end{equation}
If $\omega\in \Omega_2,$
\begin{align*}
  -\ln I_{u,v}(\xi_{u,v})&=-\ln v-\int_0^{\xi_{u,v}}[\beta S_{u,v}(\tau)-c_2]d\tau+\sigma_2 B(\xi_{u,v})
   \\&\leq -\ln v-(\beta H-c_2)\xi_{u,v}+\sigma_2B(\xi_{u,v})\leq-\ln v+\frac{\lambda}8.
\end{align*}
As a result, $\Omega_2\cap\{\xi_{u,v}<T^*\}\subset\Omega_4$. Hence $$\PP(\Omega_5)=\PP(\Omega_5\cap\{\xi_{u,v}<T^*\})+\PP\{\Omega_5\cap\{\xi_{u,v}=T^*\})\leq\PP(\Omega_1^c)+\PP(\Omega_2^c)\leq\eps.$$
Let $t<T^*$, $u'>0$ and $v'$ satisfy $-\ln v'\leq -\ln v+\frac{\lambda}8\leq0$.  In view of Proposition \ref{prop2.3} and the strong Markov property,  	 we can estimate the conditional expectation
\begin{align*}
\E\Big[[\ln I_{u,v}(2T^*)]_-^2\Big{|}&\xi_{u,v}=t, \;I_{u,v}(\xi_{u,v})=v', S_{u,v}(\xi_{u,v})=u'\Big]
\\&
\leq [\ln v']_-^2-\lambda(2T^*-t)[\ln v']_-+K_5(2T^*-t)^2
\\&
\leq [\ln v']_-^2-\lambda T^*[\ln v']_-+4K_5{T^*}^2
\\&
\leq \big(-\ln v+\frac{\lambda}8\big)^2-\lambda T^*(-\ln v)+4K_5{T^*}^2
\\&
\leq \big(-\ln v\big)^2-\big(\lambda T^*-\frac{\lambda}4\big)(-\ln v)+4K_5{T^*}^2+\frac{\lambda^2}{64}
\\&
\leq [\ln v]_-^2-\frac{3\lambda T^*}4[\ln v]_-+4K_5{T^*}^2+\frac{\lambda^2}{64}.
\end{align*}
As a result,
\begin{equation}\label{e2.43}
  \E\left(\1_{\Omega_4}[\ln I_{u,v}(2T^*)]_-^2\right)
\leq [\ln v]_-^2\PP(\Omega_4)-\frac{3\lambda T^*}4[\ln v]_-\PP(\Omega_4)+4K_5{T^*}^2+\frac{\lambda^2}{64}.
\end{equation}
In view of Lemma \ref{lem2.4},
\begin{equation}\label{e2.44}
  \E\left(\1_{\Omega_5}[\ln I_{u,v}(2T^*)]_-^2\right)
\leq [\ln v]_-^2\PP(\Omega_5)+K_3\sqrt{\PP(\Omega_5)}2T^*[\ln v]_-+4K_4{T^*}^2.
\end{equation}
Adding side by side \eqref{e2.42}, \eqref{e2.43}, \eqref{e2.44}, we have
\begin{align}
\E\left([\ln I_{u,v}(2T^*)]_-^2\right)
&\leq [\ln v]_-^2-T^*\left(\frac{3\lambda}{4}(1-\eps)+2K_3\sqrt{\eps}\right)+K_7{T^*}^2
\notag\\&
\leq
[\ln v]_-^2-\frac{\lambda T^*}{2}+K_7{T^*}^2
  \end{align}
for some $K_7>0.$
We note that, if $v\geq \exp\{-\frac{\lambda T^*}2\}$ then $-\ln v\leq \frac{\lambda T^*}{2}.$
Therefore, it follows from Lemma \ref{lem2.4} that
$$
\E\left([\ln I_{u,v}(2T^*)]_-^2\right)
\leq\Big(\frac{\lambda^2 }{4}+K_3\frac{\lambda }2+4K_4\Big){T^*}^2.
 $$
Let $K_6=\max\{K_7, \frac{\lambda^2}{4}+K_3\frac{\lambda}2+4K_4\},$ we have
\begin{equation*}
\E\left([\ln I_{u,v}(2T^*)]_-^2\right)\leq [\ln v]_-^2-\frac{\lambda T^*}2[\ln v]_-+K_6{T^*}^2\;\; \forall u\geq H, \; v\in(0, \infty).
 \end{equation*}
The proof is complete.
 \end{proof}

 \begin{lem}\label{petite}
  Any compact set $K\subset\R^{2,\circ}_+$ is petite for the Markov chain $(S_{u,v}(2nT^*), I_{u,v}(2nT^*))$ $(n\in\N)$.
The irreducibility and aperiodicity of $(S_{u,v}(2nT^*), I_{u,v}(2nT^*))$ $(n\in\N)$ is a byproduct (see \cite{EN, MT}).
\end{lem}

\begin{proof}
Note that, we can always choose $\phi_*\in\R$ such that $(c_1-\sigma_1\phi_*)>0$, $(c_2-\sigma_2\phi_*)>0$
and $\alpha\beta-(c_1-\sigma_1\phi_*)(c_2-\sigma_2\phi_*)>0$.
Hence, with the constant control $\phi_*$
we can show that the solution to \eqref{e3.2} with control $\phi_*$, $(u_{\phi_*}(t, u,v), v_{\phi_*}(t, u, v))$, converges to a positive equilibrium $(u_*, v_*)$ for all $(u,v)\in\R^{2,\circ}_+$.
Let $(u_\diamond,v_\diamond)\in\R^{2,\circ}_+$ such that
$p(2T^*, u_*, v_*, u_\diamond,v_\diamond)>0$.
By the smoothness of  $p(2T^*,\cdot,\cdot,\cdot,\cdot)$, there exists  a  neighborhood $W_\delta=(u_*-\delta,u_*+\delta)\times(v_*-\delta,v_*+\delta)$, that is invariant under \eqref{e3.2} and  a  open set $G\ni(u_\diamond,v_\diamond)$ such that
\begin{equation}\label{bs1}
p(1, u,v,u',v')\geq m'>0  \;\forall\, (u,v)\in W_\delta, (u',v')\in G.
\end{equation}
Since $(u_{\phi_*}(t, u,v), v_{\phi_*}(t, u, v))$ converges to a positive equilibrium $(u_*, v_*)$,
in view of the support theorem (see \cite[Theorem 8.1, p. 518]{IW}), there is $n_{u,v}\in\N$ such that
$$P(2n_{u,v}T^*,u,v, W_\delta):=2\rho_{u,v}>0.$$
Since $(S_{u,v}(t), I_{u,v}(t))$ is a Markov-Feller process,  there exists a open set $V_{u,v}\ni(u,v)$ such that
$P(n_{u,v},u',v', W_\delta)\geq \rho_{u,v} \;\forall (u',v')\in V_{u,v}.$
Since $K$ is a compact set,   there is a finite number of $V_{u_i,v_i}, \; i=1, \ldots, l$ satisfying $K\subset\bigcup_{i=1}^lV_{u_i, v_i}.$
Let $\rho_K=\min\{\rho_{u_i,v_i}, \; i=1, \ldots, l\}.$
For each $(u,v)\in K$, there exists $n_{u_i,v_i}$ such that
\begin{equation}\label{bs2}
  P(n_{u_i,v_i}, u,v, W_\delta)\geq\rho_K.
\end{equation}
From \eqref{bs1} \eqref{bs2}, for all $(u,v)\in K$ there exists $n_{u_i,v_i}$ such that
\begin{equation}\label{bs3}
  p((2n_{u_i,v_i}+2)T^*,u,v,u',v')\geq \rho_Km' \;\forall\, (u',v')\in G.
\end{equation}
Define the kernel
$$\K(u,v,Q):=\frac1l\sum_{i=1}^lP((2n_{u_i,v_i}+2)T^*,u,v, Q) \;\forall Q\in\mathcal B(\R^{2,\circ}_+).$$
We derive from \eqref{bs3} that
\begin{equation}\label{bs4}
  \K(u,v,Q)\geq \frac1l\rho_Km'\mu(G\cap Q)  \;\forall Q\in\mathcal B(\R^{2,\circ}_+),
\end{equation}
where $\mu(\cdot)$ is the Lebesgue measure on $\R^{2,\circ}_+.$
\eqref{bs4} means that every compact set $K\subset \R^{2,\circ}_+$  is petite for the Markov chain $(S_{u,v}(2T^*n), I_{u,v}(2T^*n)).$
\end{proof}
\begin{thm}\label{thm2.2}
Let $\lambda>0$, $d^*$ as in Lemma \ref{lem2.3}. There exists an invariant probability measure $\pi^*$ such that
\begin{equation}\label{etv}
\lim\limits_{t\to\infty} t^{q^*}\|P(t, (u,v), \cdot)-\pi^*(\cdot)\|=0 \;\forall(u,v)\in\R^{2,\circ}_+,
\end{equation}
where $\|\cdot\|$ is the total variation norm and $q^*$ is any positive number.
The support of $\pi^*$  is $\R^{2\circ}_+$ if $d^*\leq0$ and  is $\{(u,v)\in\R^{2,\circ}_+: u^rv\geq d^*\}$ if $d^*>0$.
Moreover, for any initial value $(u,v)\in\R^{2,\circ}_+$ and a $\pi^*$-integrable function $f$ we have
\begin{equation}\label{slln}
\PP\Big\{\lim\limits_{T\to\infty}\dfrac1T\int_0^Tf\big(S_{u,v}(t), I_{u,v}(t)\big)dt=\int_{\R_+^{2,\circ}}f(u',v')\pi^*(du', dv')\Big\}=1.
\end{equation}
\end{thm}
\begin{proof}
By virtue of Lemma \ref{lem2.1}, there are $h_1>0, H_1>0$ satisfying
\begin{equation}\label{e0.3}
\E U\Big(S_{u,v}(2T^*),I_{u,v}(2T^*)\Big)\leq (1-h_1)U(u,v)+H_1\,\,\forall (u,v)\in\R^{2,\circ}_+.
\end{equation}
Let $V(u,v)=U(u,v)+[\ln v]_-^2$.
In view of Propositions \ref{prop2.3}, \ref{prop2.4} and \eqref{e0.3},
there is a compact set $K\subset\R^{2,\circ}_+$, $h_2>0, H_2>0$ satisfying
\begin{equation}\label{e0.4}
\E V(S_{u,v}(2T^*),I_{u,v}(2T^*))\leq V(u,v)-h_2\sqrt{V(u,v)}+H_2\1_{\{(u,v)\in K\}}\,\forall (u,v)\in\R^{2,\circ}_+.
\end{equation}
Applying \eqref{e0.4} and Lemma \ref{petite} to \cite[Theorem 3.6]{JR} we obtain that
\begin{equation}\label{e0.5}
n\|P(2nT^*,(u,v),\cdot)-\pi^*\|\to 0 \text{ as } n\to\infty,
\end{equation}
for some invariant probability measure $\pi^*$ of the Markov chain $(S_{u,v}(2nT^*), I_{u,v}(2nT^*))$.
It is shown in the proof of \cite[Theorem 3.6]{JR} that \eqref{e0.4} implies $\E \tau_K<\infty$ where $\tau_K=\inf\{n\in\N: (S_{u,v}(2nT^*), I_{u,v}(2nT^*))\in K\}$.
In view of \cite[Theorem 4.1]{WK},
the Markov process $(S_{u,v}(t), I_{u,v}(t))$ has an invariant probability measure $\pi_*$.
As a result, $\pi_*$ is also an invariant probability measure of the Markov chain $(S_{u,v}(2nT^*), I_{u,v}(2nT^*))$.
In light of \eqref{e0.5}, we must have $\pi_*=\pi^*$, that is, $\pi^*$ is an invariant measure of
the Markov process $(S_{u,v}(t), I_{u,v}(t))$.

In the proofs, we use the function $[\ln v]_-^2$ for the sake of simplicity.
In fact, we can treat $[\ln v]_-^{1+q}$ for any small $q\in(0,1)$ in the same manner.
We can show that there is $h_q, H_q>0$, a compact set $K_q$ satisfying
\begin{equation}\label{e0.6}
\E V_q(S_{u,v}(2T^*),I_{u,v}(2T^*))\leq V_q(u,v)-h_q[V_q(u,v)]^{\frac1{1+q}}+H_q\1_{\{(u,v)\in K_p\}}\,\forall (u,v)\in\R^{2,\circ}_+,
\end{equation}
where $V_q(u,v)=U(u,v)+[\ln v]_-^{1+q}.$
Then applying \cite[Theorem 3.6]{JR}, we can obtain
$$n^{1/q}\|P(2nT^*,(u,v),\cdot)-\pi\|\to 0 \text{ as } n\to\infty.$$
Since $\|P(t,(u,v),\cdot)-\pi^*\|$ is decreasing in $t$, we easily deduce
$$t^{q^*}\|P(t,(u,v),\cdot)-\pi^*\|\to 0 \text{ as } t\to\infty$$
where $q^*=1/q\in(1,\infty)$.

On the one hand, in view of Lemma \ref{lem2.3}, the invariant control set of \eqref{e3.2}, says $\cC$, is $\R^{2,\circ}$ if $d^*\leq0$ and $\{(u,v)\in\R^{2,\circ}_+: u^rv\geq d\}$ if $d^*>0$.
By \cite[Lemma 4.1]{WK},  $\cC$ is exactly the support of the unique invariant measure $\pi^*$.
The strong law of large number can be obtained by using \cite[Theorem 8.1]{MT2} or \cite{WK}.
\end{proof}

\section{Discussion and Numerical Examples}\label{sec:ex}
\subsection{Discussion}
We have shown that the extinction and permanence of the disease in a stochastic SIR model can be determined by the sign of a threshold value $\lambda$. Only the critical case $\lambda=0$ is not studied in this paper.
To illustrate the significance of our results, let us compare our results with those in  \cite{LJX}.

\begin{thm}\textup {\cite[Theorem 3.1]{LJX}}\label{thm3.1bs} Assume that $\sigma_1>0, \sigma_2>0$. Let $(S_{u,v}(t), I_{u,v}(t))$  be a solution of system
\eqref{e2.11ss}.
If $\mu>\sigma_1^2, \mu+\rho+\gamma>\sigma_2^2, R_0>1$ and
$$\delta<\min\left\{\frac{\mu^2}{\mu-\sigma_1^2}{S^*}^2,  \frac{(\mu+\rho+\gamma)^2}{\mu+\rho+\gamma-\sigma_2^2}{I^*}^2 \right\}.$$
Then there exists a stationary distribution $\pi^*$ for the Markov process $(S_{u,v}(t), I_{u,v}(t))$ which is the limit in total variation of transition probability $ P(t, (u,v),\cdot)$. Here
$$\delta=\frac{\mu\sigma_1^2}{\mu-\sigma_1^2}{S^*}^2
+\frac{(\mu+\rho+\gamma)\sigma_2^2}{\mu+\rho+\gamma-\sigma_2^2}{I^*}^2
+\frac{(\mu+\rho+\gamma}{2\beta}{I^*}\sigma_2^2,$$
$$S^*=\frac{\mu+\rho+\gamma}{\beta},\; I^*=\frac{\alpha}{\mu+\rho+\gamma}-\frac{\mu}{\beta}; \;\;R_0=\frac{\beta\alpha}{\mu(\mu+\rho+\gamma)}.$$
\end{thm}

By straightforward calculations or by arguments in Section 4 of \cite{DDY} we can show that their conditions are much more restrictive than the condition $\lambda>0$.
Moreover, it should be noted that Theorem \ref{thm2.1}  is the same as Lemma 3.5 in \cite{LJX}.
In contrast to the aforementioned paper, we provide a  rigorous proof of
Theorem \ref{thm2.1} here.
Moreover, the conclusions in Theorems \ref{thm2.1} and \ref{thm2.2} still hold for the non-degenerate model \eqref{e1.1a}.
As a result we have the following theorem for
model \eqref{e1.1a}.

\begin{thm}\label{thm3.2}
Let $(S_{u,v}(t), I_{u,v}(t))$ be the solution to \eqref{e1.1a} with initial value $(S(0), I(0))=(u,v)\in\R^{2,\circ}_+$. Define $\lambda$ as \eqref{lambda}.
If $\lambda<0$, then $\lim\limits_{t\to\infty}I_{u,v}(t)=0$ a.s. and the distribution of $S_{u,v}(t)$ converges weakly to
 $\mu^*$, which has the density \eqref{density}.
If $\lambda>0$, the solution process $(S_{u,v}(t), I_{u,v}(t))$ has a unique invariant probability measure $\varphi^*$ whose support is $\R^{2,\circ}_+$. Moreover, the transition probability $P(t, (u,v),\cdot)$ of  $(S_{u,v}(t), I_{u,v}(t))$ converges to $\varphi^*(\cdot)$ in total variation. The rate of convergence is bounded above by any polynomial rate. Moreover, for any $\varphi^*$-integrable function
$f$, we have
\begin{equation*}
\PP\Big\{\lim\limits_{t\to\infty}\dfrac1t\int_0^tf\big(S_{u,v}(\tau), I_{u,v}(\tau)\big)d\tau=\int_{\R^{2,\circ}_+}f(u',v')\varphi^*(du', dv')\Big\}=1 \  \ \forall (u,v)\in\R^{2,\circ}_+.
\end{equation*}
 \end{thm}

It should be emphasized that our techniques can be also used to improve results in \cite{CKBW,JJS, YWS}.

\subsection{Example}
 Let us finish this paper by providing some numerical examples.

\begin{example}\label{ex1}
Consider \eqref{e1.1} with parameters $\alpha=20$, $\beta=4$, $\mu=1$, $\rho=10$, $\gamma=1$, $\sigma_1=1$, and  $\sigma_2=-1.$
Direct calculation shows that $\lambda=67.5>0$, $d^*=7.75>0$, and $c^*=1.9375.$
By virtue of Theorem \ref{thm2.2},
\eqref{e1.1} has a unique invariant probability measure $\pi^*$ whose support is
    $\{(S, I): S\geq \frac{1.9375}{I}\}$.
Consequently, the strong law of large numbers and the convergence in total variation norm of the transition probability hold.
A sample path of solution to \eqref{e1.1} is illustrated by Figures \ref{f1.1s}, while the phase portrait in Figure \ref{f1.1ss} demonstrates that the support of $\pi^*$ lies above and includes the curve $S=\frac{c^*}{I}=\frac{1.9375}{I}$ as well as  the empirical  density of   $\pi^*$. In non-degenerate case (Eq. \eqref{e1.1a}), with this same set of parameters  the empirical  density of  $\pi^*$ is illustrated by Figure \ref{f1.1sss}.

   \begin{figure}[b]
\centering
\includegraphics[width=7.5682cm, height=5cm]{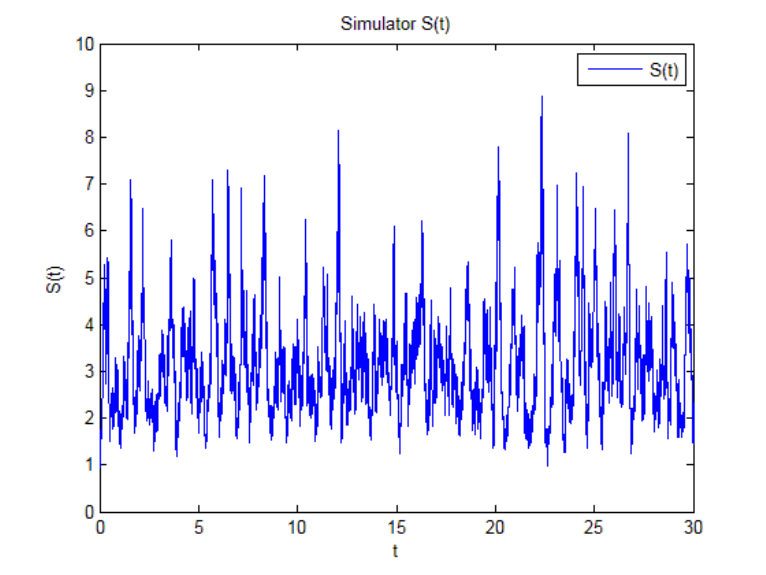}
\includegraphics[width=7.5682cm, height=5cm]{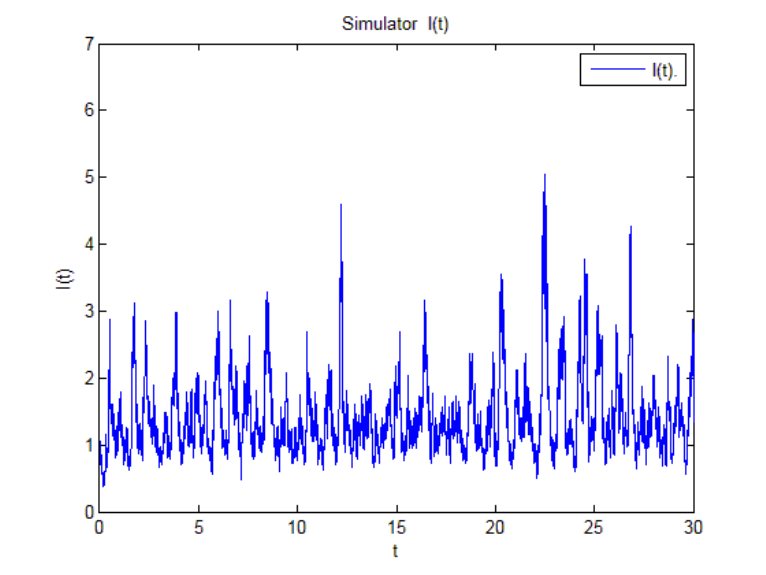}
\caption{Trajectories of $S_{u,v}(t), I_{u,v}(t)$ in Example \ref{ex1}.}\label{f1.1s}
\end{figure}

  \begin{figure}[p]
\centering
\includegraphics[width=7.5cm, height=5.5cm]{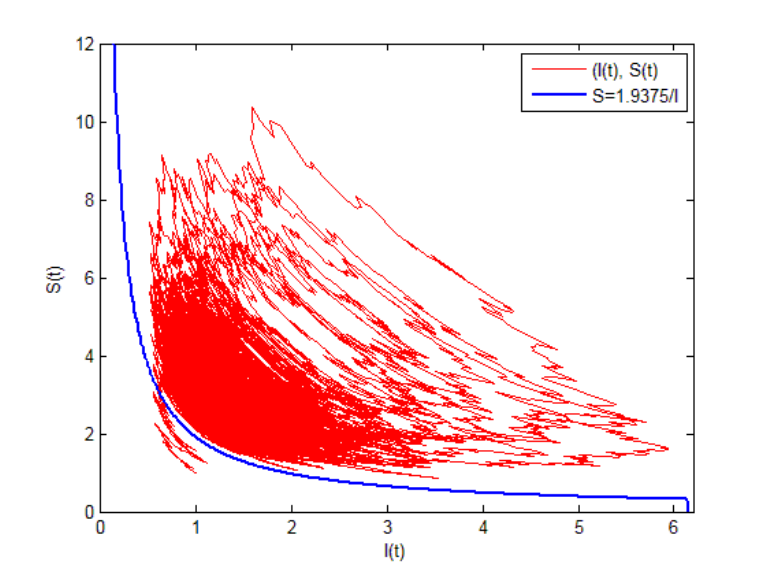}
\includegraphics[width=7.5cm, height=5.5cm]{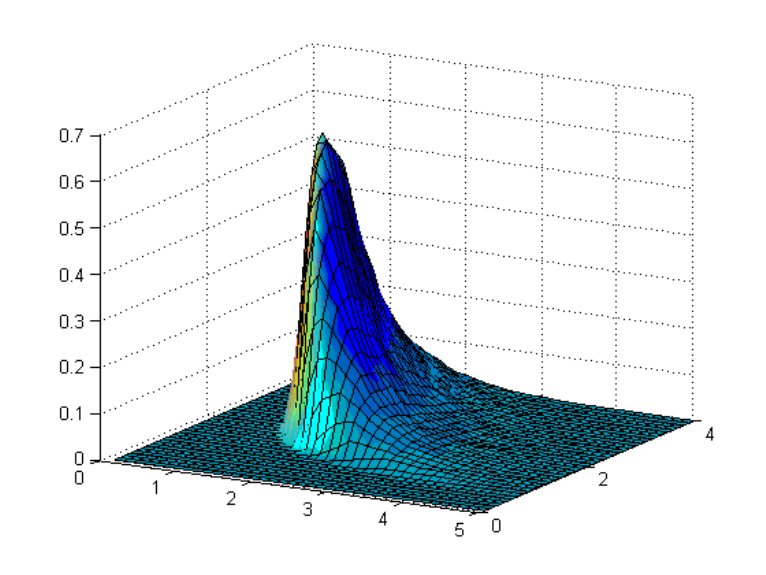}
\caption{Phase portrait of \eqref{e1.1};   the boundary $S=\frac{1.9375}{I}$ of the support of $\pi^*$  and the empirical  density of   $\pi^*$ in Example \ref{ex1}.}\label{f1.1ss}
\end{figure}

  \begin{figure}[p]
\centering
\includegraphics[width=8.5cm, height=6.5cm]{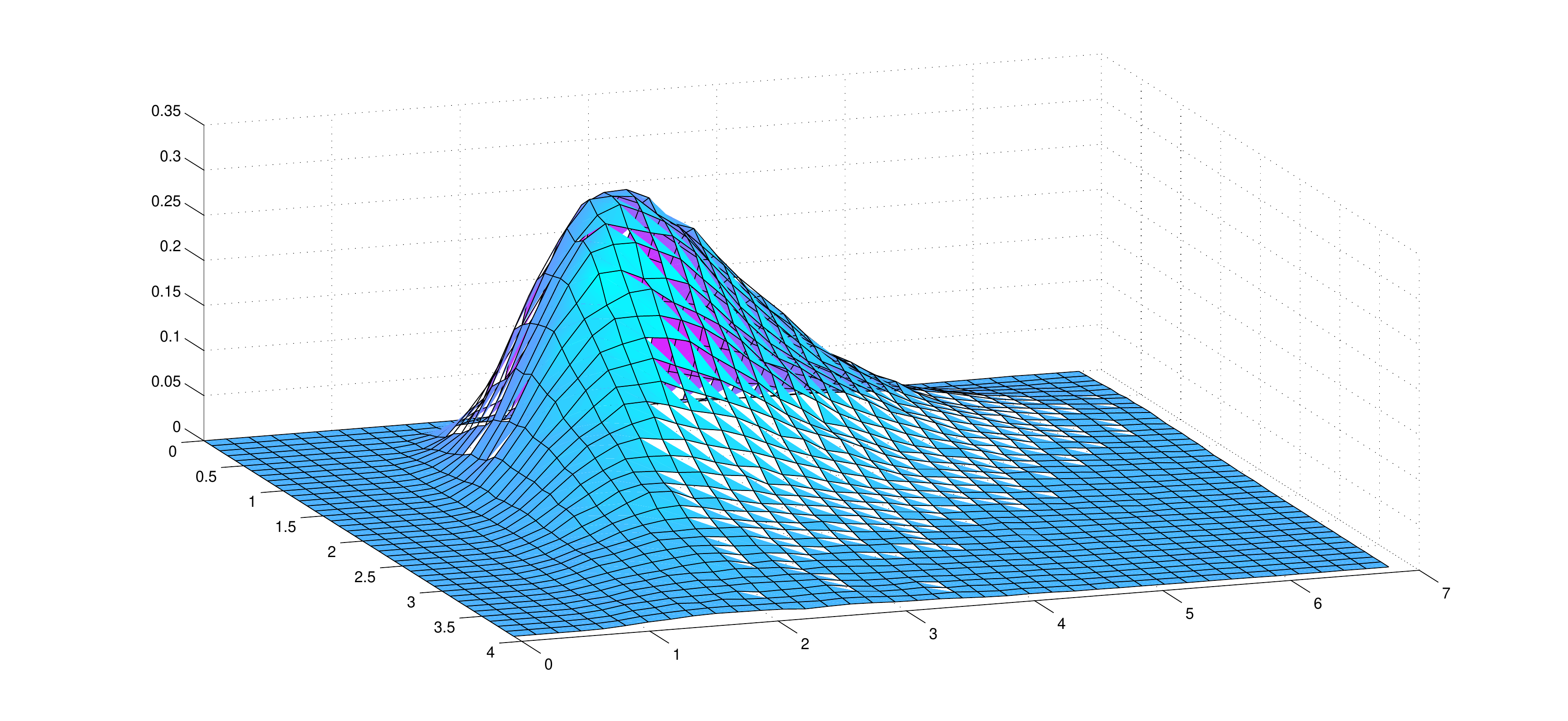}
\caption{The empirical  density of   $\varphi^*$ in Example \ref{ex1} for the non-degenerate equation \eqref{e1.1a}.}\label{f1.1sss}
\end{figure}
\end{example}

\begin{example}\label{ex2}
Consider \eqref{e1.1} with parameters $\alpha=7$, $\beta=3$, $\mu=1$, $\rho=1$, $\gamma=2$,
  $\sigma_1=1$, and $\sigma_2=1.$ For these parameters, the conditions in Theorem \ref{thm3.1bs} are not satisfied.
We obtain $\lambda=16.5>0, d^*=-\infty.$
       As a result of Theorem \ref{thm2.2},
        \eqref{e1.1} has a unique invariant probability measure $\pi^*$ whose support is
    $\R^{2,\circ}_+$.
Consequently, the strong law of large numbers and the convergence in total variation norm of the transition probability hold.
A sample path of solution to \eqref{e1.1} is
 depicted in Figures \ref{f1.2}, while the phase portrait in Figure \ref{f1.22} demonstrates that the support of $\pi^*$ and the empirical  density of   $\pi^*$.

  \begin{figure}[p]
\centering
\includegraphics[width=7.5682cm, height=5cm]{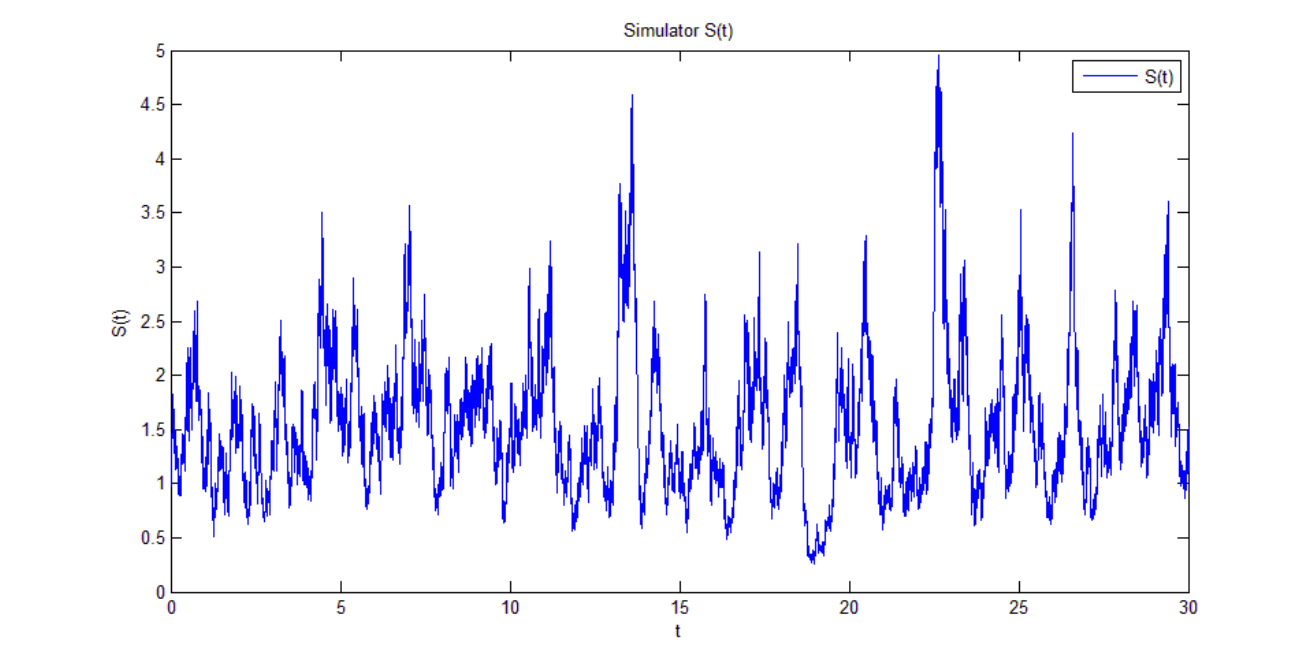}
\includegraphics[width=7.5682cm, height=5cm]{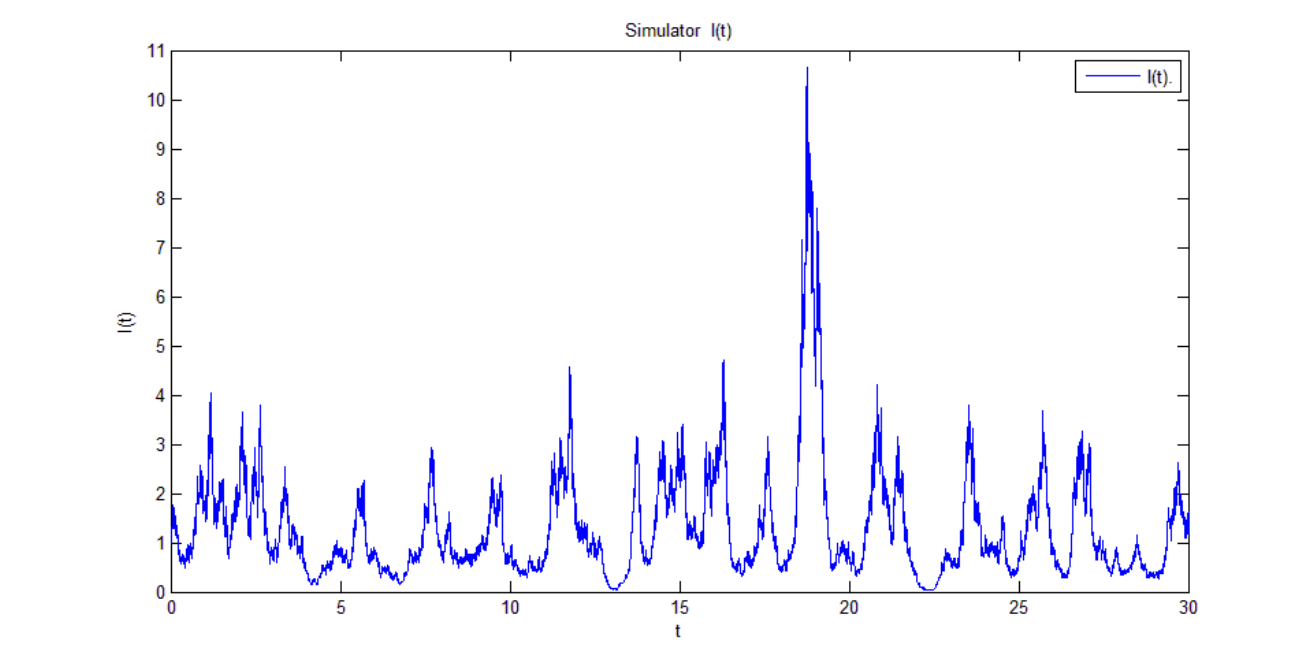}
\caption{ Trajectories of $S_{u,v}(t), I_{u,v}(t)$ in Example \ref{ex2}.}\label{f1.2}
\end{figure}

  \begin{figure}[p]
\centering
\includegraphics[width=8.868cm, height=5cm]{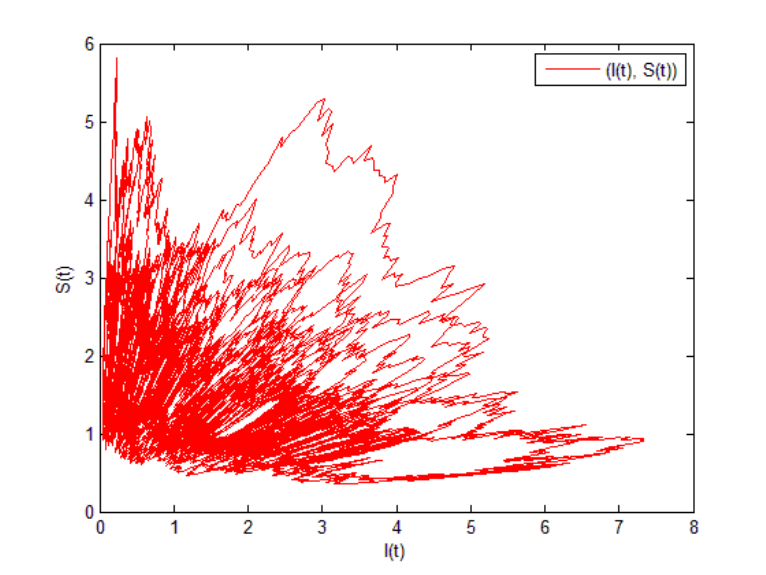}
\includegraphics[width=7cm, height=6cm]{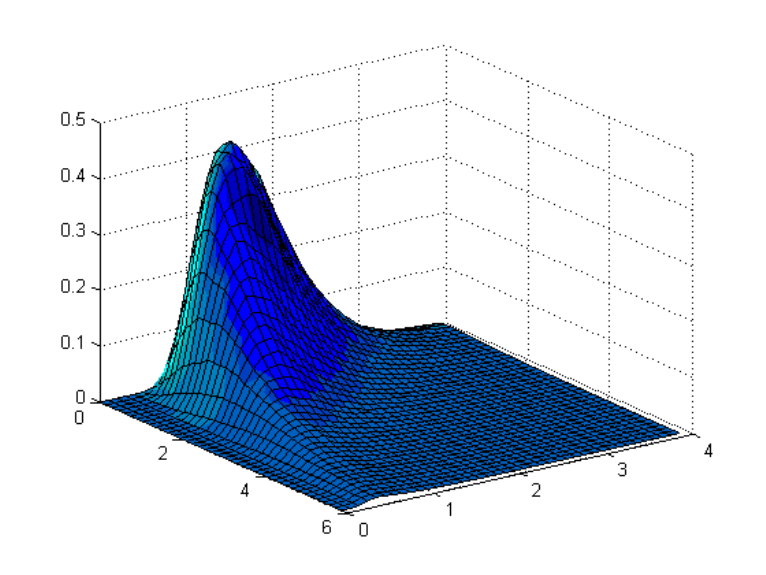}
\caption{Phase portrait of \eqref{e1.1} and  the empirical  density of   $\pi^*$  in Example \ref{ex2}.}\label{f1.22}.
\end{figure}
\end{example}

\begin{example}\label{ex3}
Consider \eqref{e1.1} with parameters $\alpha=5$, $\beta=5$, $\mu=4$, $\rho=1$, $\gamma=1$, $\sigma_1=2,$
and $\sigma_2=-1.$
It can be shown that $\lambda=-1.75<0$
 As a result of Theorem \ref{thm2.1},
  $I_{u,v}(t)\to0$ a.s. as $t\to\infty$. This claim is
  supported by  Figures \ref{f4.1}.
  That is, the population will eventually have no  disease. The distribution of  $S_{u,v}(t)$ convergence to $f^*(x)$ as $t\to\infty.$ The graphs of $f^*(x)$ and empirical  density of $S_{u,v}(t)$ at $t=50$ are illustrated by Figure \ref{f4.2}.

  \begin{figure}[p]
\centering
\includegraphics[width=7.5682cm, height=5cm]{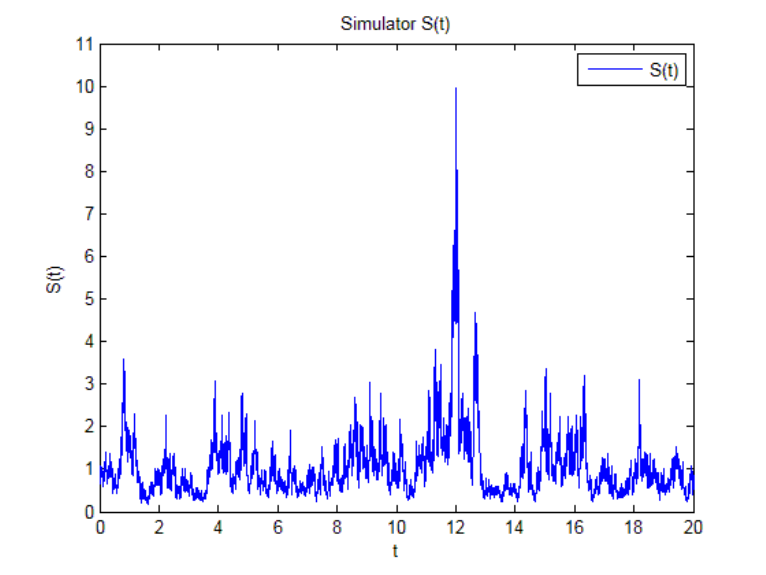}
\includegraphics[width=7.5682cm, height=5cm]{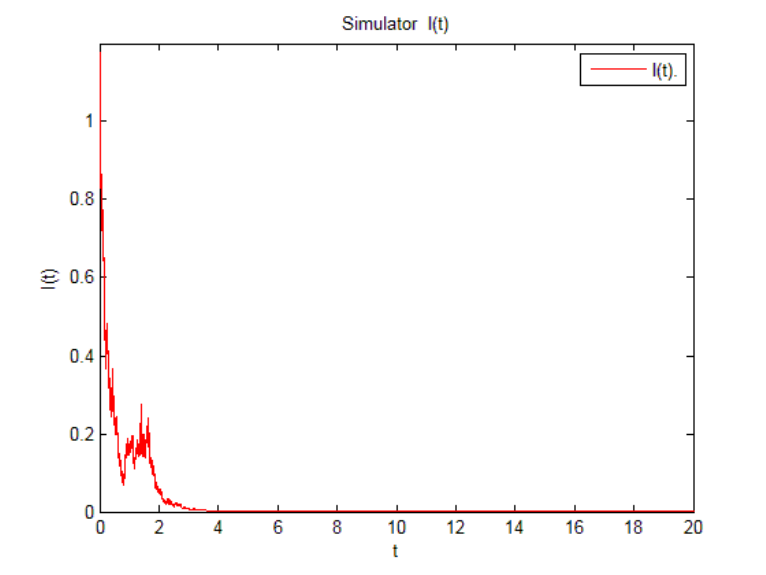}
\caption{ Trajectories of $S_{u,v}(t), I_{u,v}(t)$ in Example \ref{ex3}.}\label{f4.1}
\end{figure}
  \begin{figure}[htp]
\centering
\includegraphics[width=10cm, height=7cm]{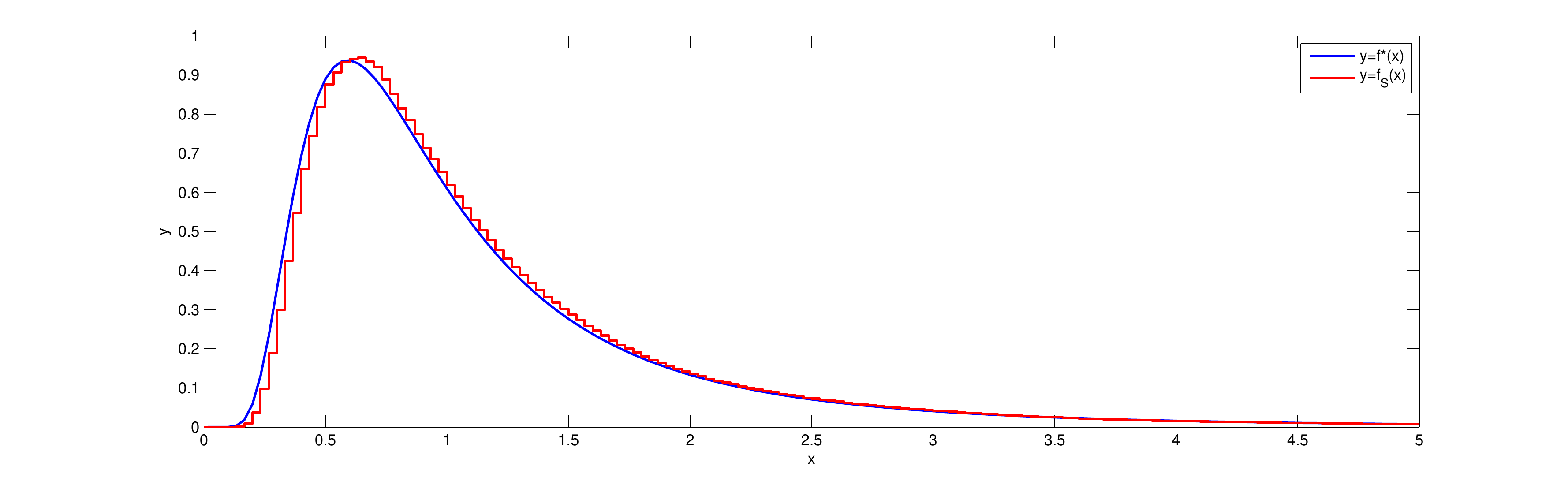}
\caption{The graph of the stationary density $f^*$ (in blue) and the graph of  the empirical  density of $S(50)$ (in red) in Example \ref{ex3}.}\label{f4.2}
\end{figure}

\end{example}

\end{document}